\documentclass[11pt]{amsart}
\usepackage{amsmath,amsthm,amssymb,url}
\usepackage[all]{xy}
\usepackage[pagebackref]{hyperref}

\newtheorem*{theorems}{Theorem}
\newtheorem{theorem}{Theorem}[section]
\newtheorem{lemma}[theorem]{Lemma}
\newtheorem{proposition}[theorem]{Proposition}
\newtheorem{corollary}[theorem]{Corollary}

\theoremstyle{definition}
\newtheorem{definition}[theorem]{Definition}
\newtheorem{example}[theorem]{Example}

\theoremstyle{remark}
\newtheorem{remark}[theorem]{Remark}

\newcommand{\de}{\partial}

\newcommand{\bi}{\boldsymbol{i}}
\newcommand{\bl}{\boldsymbol{l}}

\renewcommand{\bar}{\overline}

\newcommand{\debar}{{\overline{\partial}}}

\newcommand{\sA}{\mathcal{A}}
\newcommand{\Oh}{\mathcal{O}}
\newcommand{\sF}{\mathcal{F}}

\newcommand{\sL}{\mathcal{L}}

\newcommand{\g}{\mathfrak{g}}

\newcommand{\Z}{\mathbb{Z}}
\newcommand{\C}{\mathbb{C}}
\newcommand{\K}{\mathbb{K}}

\newcommand{\Id}{\operatorname{Id}}

\newcommand{\Spec}{\operatorname{Spec}}

\newcommand{\MC}{\operatorname{MC}}
\newcommand{\Def}{\operatorname{Def}}
\newcommand{\Hom}{\operatorname{Hom}}

\newcommand{\Der}{\operatorname{Der}}

\newcommand{\coker}{\operatorname{Coker}}
\newcommand{\tot}{\operatorname{Tot}}

\newcommand{\Art}{\mathbf{Art}}
\newcommand{\Set}{\mathbf{Set}}

 \newcommand{\contr}{{\mspace{1mu}\lrcorner\mspace{1.5mu}}}

\newenvironment{acknowledgement}{\par\addvspace{17pt}\small\rm
\trivlist\item[\hskip\labelsep{\it Acknowledgement.}]}
{\endtrivlist\addvspace{6pt}}

\begin{document}

\title{Deformations and obstructions of  pairs $(X,D)$}

\author{Donatella Iacono}

\address{\newline
Department of Mathematics,\hfill\newline
Imperial College London}
\email{d.iacono@imperial.ac.uk}
\address{\newline Dipartimento di Matematica
\newline  Universit\`a degli Studi di Bari}
\email{donatella.iacono@uniba.it}

\begin{abstract}


We study infinitesimal deformations of pairs  $(X,D)$ with $X$ smooth projective variety and $D$ a smooth or a normal crossing divisor, defined over an  algebraically closed field of  characteristic $0$.
Using the differential graded Lie algebras theory and the Cartan homotopy construction, we are able to prove in a completely algebraic way  the unobstructedness of the deformations of the pair  $(X,D)$  in many cases, e.g., 
whenever $(X,D)$ is a log Calabi-Yau  pair,  in the case of a smooth divisor $D$ in a  Calabi Yau variety $X$ and when $D$ is  a smooth divisor in $|-m K_X|$, for some positive integer $m$.

\end{abstract}

\maketitle

\section*{Introduction}

Let $X$ be a smooth projective variety over an algebraically
closed field $\K$ of characteristic 0. If $X$ has trivial canonical bundle (torsion is enough), then the deformations of  $X$ are unobstructed: this is
the well known Bogomolov-Tian-Todorov theorem. The first proofs of this theorem, due to  Bogolomov in \cite{bogomolov}, Tian \cite{Tian} and Todorov \cite{Todorov}, are trascendental and rely on the underlying differentiable structure of the variety $X$. More algebraic proof,  based on  the $T^1$-lifting theorem and the degeneration of the Hodge-to-de Rham spectral sequence, are due to Ran \cite{zivran}, Kawamata \cite{Kaw1} and Fantechi-Manetti \cite{FM2}.
The Bogomolov-Tian-Todorov  theorem is also a consequence  of the stronger fact that the differential graded Lie algebra associated with the infinitesimal deformations of $X$ is homotopy abelian,  i.e.,
quasi-isomorphic to an abelian differential graded Lie algebra.
For $\K=\mathbb{C}$,  this was first proved in  \cite{GoMil2}, see also \cite{manRENDICONTi}.
For any algebraically closed field $\K$ of characteristic 0, this was proved in a completely algebraic way in \cite{algebraicBTT}, using the degeneration of the Hodge-to-de Rham spectral sequence and the notion of Cartan homotopy.

The aim of this paper is then to extend these techniques to analyse the infinitesimal deformations of pairs; indeed, we prove that the DGLA associated with these deformations   is homotopy abelian in many cases, and hence the deformations are unobstructed.
This extension can be viewed as an application of the Iitaka's philosophy: \lq\lq whenever we have a theorem about non singular complete varieties whose statement is dictated by the behaviour of the regular differentials forms (the canonical bundles), there should exist a corresponding theorem about logarithmic paris (pairs consisting of nonsingular complete varieties and boundary divisors with only normal crossings) whose statement is dictated by the behaviour of the logarithmic forms (the logarithmic canonical bundles) and vice versa\rq\rq \cite[Principle 2-1-4]{matsuki}.

More precisely, let   $X$ be a smooth projective variety, $D$ a smooth divisor and consider the deformations of the pair $(X,D)$, i.e., the deformations of the closed embedding $ j: D \hookrightarrow X$.  As first step, we give an explicit description of a differential graded Lie algebra controlling the deformations of $j$. Namely, let  ${\Theta}_X(-\log D)$ be the sheaf of germs of the tangent vectors to $X$ which are tangent to $D$. Once we fix an open affine cover $\mathcal{U}$ of $X$, the Thom-Whitney construction applied to ${\Theta}_X(-\log D)$ provides a differential graded Lie algebra  $TW( {\Theta}_X(-\log D)(\mathcal{U}))$ controlling  the deformations of $j$ (Theorem \ref{teo. DGLA controlling def of j}).
If the ground field is $\C$, then we can simply take  the DGLA associated with the Dolbeault resolution of ${\Theta}_X(-\log D)$, i.e.,  $A^{0,*}_X({\Theta}_X(-\log D))= \oplus_i \Gamma(X, \sA^{0,i}_X({\Theta}_X(-\log D)))$ (Example \ref{exa DGLA (X,D) on C}).

Then, we provide  a condition that ensures that the DGLA $TW( {\Theta}_X(-\log D)(\mathcal{U}))$ is homotopy abelian.

\begin{theorems}
Let  $X$ be a smooth projective variety  of dimension $n$, defined
over an algebraically closed field of characteristic 0 and  $D\subset X$  a smooth divisor D. If the contraction map
\[
H^*(X,\Theta_X (-\log D))\xrightarrow{\bi}\Hom^*(H^*(X,\Omega^n_X (\log D)),H^*(X,\Omega^{n-1}_X  (\log D)))\]
is injective, then the DGLA $TW( {\Theta}_X(-\log D)(\mathcal{U}))$ is homotopy abelian, for every affine open cover $\mathcal{U}$ of $X$.
\end{theorems}

As in \cite{algebraicBTT}, we recover this result  using  the power of the  Cartan homotopy construction and the degeneration of the Hodge-to-de Rham spectral sequence associated in this case  with the complex  of  logarithmic differentials $\Omega^\ast_X (\log D)$.

As corollary, we obtain an alternative (algebraic) proof,   that, in the case of a log Calabi-Yau pair (Definition \ref{definiiton log calabi yau}),  the DGLA controlling the infinitesimal deformations of the pair $(X,D)$ is homotopy abelian (Corollary \ref{corollari log calabi yau formal smooth}). In particular, we are able to prove the following result  about smoothness of deformations (Corollary \ref {cor.log calabi yau no obstruction}).

\begin{theorems}
Let  $X$  be a smooth  projective $n$-dimensional variety defined over an algebraically closed field of characteristic 0 and  $D \subset X$ a smooth divisor. If $(X,D)$ is a log Calabi-Yau pair, i.e., the logarithmic canonical bundle $\Omega^n_X (\log D)\cong \Oh(K_X+D)$ is trivial,  then the pair  $(X,D)$ has unobstructed deformations.
\end{theorems}
The unobstructedness of the deformations of a log Calabi-Yau pair  $(X,D)$ is also interesting from the point of view of mirror symmetry.
The deformations of the log Calabi Yau pair $(X,D)$ should be mirror to the deformations of the (complexified) symplectic form on the mirror Landau-Ginzburg model. Therefore, these deformations are  also smooth \cite{Auroux, Auroux2,KKP}.

\smallskip

Then, we focus our attention on the deformations of pairs  $(X,D)$, with $D$ is a smooth divisor in a smooth projective Calabi Yau variety $X$. Also in this case, we provide  an alternative (algebraic) proof   that   the DGLA controlling these infinitesimal deformations   is homotopy abelian (Theorem \ref{theorem smoothness of  D inside CY})). We also show the following statement about smoothness of deformations (Corollary \ref{cor. D in calabi yau no obstruction}).

\begin{theorems}
Let  $X$  be a smooth  projective  Calabi Yau variety  defined over an algebraically closed field of characteristic 0 and  $D \subset X$ a smooth divisor. Then, the pair  $(X,D)$ has unobstructed deformations.
  \end{theorems}

The previous results are  also sketched in \cite{KKP},  see also \cite{zivran,kinosaki}, where the authors work over the field of complex number and make a deep use of transcendental methods.
More precisely, using Dolbeault type complexes, one can construct  a differential Batalin-Vilkovisky algebra such that the associated  DGLA controls the deformation problem (Definition \ref{def dbv}). If the differential Batalin-Vilkovisky algebras has a degeneration property  then the associated DGLA is  homotopy abelian \cite {terilla, KKP, BraunLaza}.
Using our approach and the  powerful notion of  the Cartan homotopy,  we are able to give an alternative    proof  of this result (Theorem  \ref{theorem dbv degener implies homotopy abelian}).

\medskip

In a very recent preprint \cite{Sano},  the $T^1$-lifting theorem is applied  in order to prove the unobstructedness  of the deformations $(X,D)$, for $X$ smooth projective variety and  $D$ a smooth divisor  in $|-m K_X|$, for some positive integers m, under the assumption $H^1(X,\Oh_X)=0$ \cite[Theorem 2.1]{Sano}.
Inspired by this paper,  we also study the infinitesimal deformations of  these pairs $(X,D)$.
Using the cyclic covers of $X$ ramified over $D$, we relate the deformations of the pair $(X,D)$ with the deformations of the pair (ramification divisor, cover) and 
we show that the DGLA associated with the deformations  of the pair $(X,D)$ is homotopy abelian. In particular,
  we  can prove  the following result about smoothness  of deformations (Proposition \ref{proposition D in -mKx smooth pair}).

\begin{theorems}
Let $X$ be a smooth projective variety and $D$ a smooth divisor such that $D \in | -mK_X|$, for some positive integer $m$. Then, the   pair  $(X,D)$ has unobstructed deformations.
\end{theorems}

We refer the reader to \cite{Sano} for   examples in the Fano setting and the relation with  the unobstructedness of weak Fano manifold. 

\smallskip

Once the unobstructedness of a pair  $(X,D)$ is proved, then studying  the forgetting morphism of functors  $\phi: \Def_{(X,D)} \to \Def_X$, one can prove the unobstructedness of $\Def_X$, for instance when $D$ is stable in $X$, i.e.,  $\phi$ is smooth \cite[Definition 3.4.22]{Sernesi}.

\medskip

The paper goes as follows.  With the aim of providing a full introduction to the subject, we include 
Section \ref{section log diff} on the notion of  the logarithmic differentials  and Section \ref{section back ground DGLA} on the DGLAs, Cartan homotopy and cosimplicial constructions, such as the Thom-Whitney DGLA.
In Section \ref{section deformation}, we review the definition of the infinitesimal deformations of the pair $(X,Z)$, for any  closed subscheme  $Z \subset X$  of a smooth variety $X$, describing the DGLA controlling these deformations.
 Section \ref{section obstruction computations} is devoted to the study of obstruction and it contains the proof of the  first three theorems. 
 In  Section \ref{section cyclic cover},  we study   cyclic covers  of a smooth projective variety $X$ ramified on a smooth  divisor $D$  and  we   prove the  last theorem stated above.
 In the last section, we apply the notion of Cartan homotopy construction to the 
 the differential graded Batalin Vilkovisky algebra setting, providing a new proof of the fact that if the differential Batalin-Vilkovisky algebras has a degeneration property  then the associated DGLA is  homotopy abelian (Theorem  \ref{theorem dbv degener implies homotopy abelian}).
 

%

\medskip

\noindent{\bf{Notation.}}
Unless otherwise specified, we work over an algebraically closed field $\K$ of characteristic 0.
Throughout the paper, we also assume that
$X$ is always a  smooth projective variety over $\K$. Actually, the main ingredient of the proofs  is the degeneration at the $E_1$-level of some  Hodge-to-de Rham spectral sequences and it holds whenever X  is smooth proper over a field of characteristic 0 \cite{DI}.

 
By abuse of notation, we denote by $K_X$ both the canonical divisor and the canonical bundle of $X$.
$\Set$  denotes the category of sets (in a fixed universe)
and   $\Art $   the category of local Artinian
$\K$-algebras with residue field $\K$. Unless otherwise specified,
for every  objects  $A\in \mathbf{Art}$, we denote by
$\mathfrak{m}_A$ its maximal ideal.

\medskip


\begin{acknowledgement}
The author wish to thank Richard Thomas for  useful discussions and comments  and for pointing out the papers \cite{kinosaki} and \cite{fujino1}, and Marco Manetti  for drawing  my attention to the paper \cite{Sano} and for useful suggestions and advices, especially on Section 5. In particular, M.M. shared with the author Theorem   \ref{theorem dbv degener implies homotopy abelian}.  I also thank Taro Sano  for comments and for pointing out a mistake in a previous version.
The author is supported by the Marie Curie Intra-European Fellowship  FP7-PEOPLE-2010-IEF Proposal $N^\circ$: 273285.
\end{acknowledgement}

\section{Review of logarithmic differentials}\label{section log diff}

Let $X$ be a  smooth projective variety of dimension $n$ and $j:Z \hookrightarrow X$   a   closed embedding of a closed subscheme $Z$. We denote by ${\Theta}_X(-\log Z)$ the sheaf of germs of the tangent vectors to $X$ which are tangent to $Z$ \cite[Section 3.4.4]{Sernesi}.
Note that, denoting by $\mathcal{I}\subset \mathcal{O}_X$   the ideal sheaf of $Z$ in $X$, then $\Theta_X(-\log  Z)$ is the subsheaf of the derivations of the sheaf $\mathcal{O}_X$ preserving the ideal sheaf $\mathcal{I}$ of $Z$, i.e.,
\[
\Theta_X(-\log  Z)=\{f\in \Der(\mathcal{O}_X,\mathcal{O}_X)\mid f(\mathcal{I}) \subset\mathcal{I}\}.
\] 
\begin{remark}
If  $Z$ is smooth in $X$, then we have the short exact sequence
\[
0 \to {\Theta}_X(-\log Z) \to {\Theta}_X \to N_{Z / X} \to 0.
\]
Note also that  if the codimension of $Z$ is at least 2, then the sheaf $ {\Theta}_X(-\log Z)$ is not locally free, see also Remark \ref{remark T log dual sheaf locally free}.
\end{remark}
Next, assume to be in the divisor setting, i.e., let  $D \subset X$ be a  globally normal crossing divisor in   $X$. With the divisor assumption, we can define the sheaves of logarithimc differentials, see for instance \cite[p. 72]{deligne}, \cite{Kawamata}, \cite[Chapter 2]{librENSview} or \cite[Chapter 8]{voisin}. For any $k \leq n$, we denote by $\Omega^k_X(\log D)$   the locally free sheaf of  differential $k$-forms with logarithmic poles along $D$.
More explicitly, let $\tau: V= X-D \to X$ and $\Omega^k_X(\ast D)= \lim_{\stackrel{\to}{ \nu}}
\Omega^k_X(\nu \cdot  D)=\tau_* \Omega^k_V$. Then, $(\Omega^\ast_X(\ast D),d)$ is a complex and $(\Omega^\ast_X(\log D),d)$ is the subcomplex such that,  for evey
open $U$ in $X$, we have
\[
\Gamma(U,\Omega^k_X(\log D))=\{\alpha \in \Gamma(U,\Omega^k_X(\ast D))\ | \, \alpha \mbox{ and } d \alpha \mbox{ have simple poles along D} \}.
\]

\begin{remark}\label{remark exac sequence Omega(logD)(-D)}
For every $p$,  the following short exact sequence of sheaves  
\[
0\to \Omega^p_X (\log D) \otimes \Oh_X(-D)   \to  
  \Omega^p_X \to   \Omega^p_D \to 0
\]
is exact  \cite[2.3]{librENSview} or \cite[Lemma 4.2.4]{lazar}.
\end{remark}

\begin{example}\cite[Chapter 8]{voisin}
In the holomorphic setting,
$\Omega^k_X(\log D)$ is   the sheaf of   meromorphic differential forms $\omega$ that admit  a pole of order at most 1 along (each component) of $D$, and the same holds for $d\omega$.
Let $z_1,z_2, \ldots, z_n$ be holomorphic coordinates on an open set $U$ of $X$, in which $D \cap U$ is defined by the equation   $ z_1z_2\cdots z_r=0$. Then, $\Omega^k_X(\log D)_{\mid U}$ is a sheaf of free $\Oh_U$-modules, for which $\displaystyle \frac{dz_{i_1}}{z_{i_1}}\wedge \cdots \frac{dz_{i_l}}{z_{i_l}}\wedge dz_{j_1} \wedge \cdots \wedge dz_{j_m} $
with $i_s \leq r$, $j_s >r$ and $l+m=k$ form a basis.
\end{example}

\begin{remark}\label{remark T log dual sheaf locally free}
The sheaves of logarithmic  $k$-forms  $\Omega^k_X(\log D)= \wedge^k \Omega^1_X(\log D)$ are locally free and the sheaf  $\Theta_X (-\log D)$  is dual to the sheaf $\Omega^1_X(\log D)$, so it is in particular locally free for $D$ global normal crossing divisor.
The sheaf of logarithmic $n$-forms  $\Omega^n_X(\log D)\cong \Oh_X(K_X+D)$ is a line bundle called the \emph{logarithmic canonical bundle} for the pair $(X,D)$.

 
\end{remark}

\begin{definition}\label{definiiton log calabi yau}
A \emph{log Calabi-Yau pair} $(X,D)$ is a pair where $D$ is a smooth divisor in a smooth projective variety  $X$ 
of dimension $n$,  and the  logarithmic canonical bundle   $\Omega^{n}_X(\log D)$ is trivial.
\end{definition}

\begin{example}
Let $X$ be a smooth projective variety and $D$ an effective smooth divisor such that $D \in | -K_X|$.
Then, the sheaf $\Omega^n_X(\log D)\cong \Oh_X(K_X+D)$ is trivial, i.e., the pair $(X,D)$ is a log Calabi Yau pair.
\end{example}

The complex  $(\Omega^{\ast}_X(\log D),d)$ is equipped with the Hodge filtration, which induces a filtration on the hypercohomology $\mathbb{H}^{\ast}(X, \Omega^{\ast}_X(\log D)) $. As for the algebraic de Rham complex, the spectral sequence associated with the Hodge filtration on $\Omega^{\ast}_X(\log D)$ has its first term equal to $E_1^{p,q}=H^q(X,\Omega^{p}_X(\log D))$.
The following   degeneration properties hold.

\begin{theorem}\label{teo degen deligne} 
(Deligne) Let $X$ be a smooth proper variety and $D \subset X$ be a  globally normal crossing divisor. Then, the spectral sequence associated with the Hodge filtration
\[
E_{1}^{p,q}=H^q(X,\Omega^{p}_X(\log D))  \Longrightarrow \mathbb{H}^{p+q}( X,\Omega^{\ast}_X(\log D))
\]
degenerates at the $E_1$-level.

\end{theorem}
\begin{proof}
This is the analogous of the degeneration of the Hodge-to-de Rham spectral sequence. As in this case, there is a complete algebraic way to prove it, avoiding analytic technique, see
\cite[Section 3]{deligneII},  \cite{DI},  \cite[Corollary 10.23]{librENSview} or  \cite[Theorem 8.35]{voisin}).

\end{proof}

\begin{theorem}\label{teo degen tensor}
 Let $X$ be a smooth proper variety and $D \subset X$ be a  globally normal crossing divisor. Then, the spectral sequence associated with the Hodge filtration
\[
E_{1}^{p,q}=H^q(X,\Omega^{p}_X(\log D) \otimes \Oh_X(-D))  \Longrightarrow 
\mathbb{H}^{p+q}( X,\Omega^{\ast}_X(\log D)  \otimes \Oh_X(-D))
\]
degenerates at the $E_1$-level.

\end{theorem}
\begin{proof}
See \cite[Section 2.29]{fujino1} or \cite[Section 5.2]{fujino2}.
\end{proof}

\section{Background on DGLAs and Cartan Homotopies}\label{section back ground DGLA}

\subsection{DGLA}

A \emph{differential graded Lie algebra}  is
the data of a differential graded vector space  $(L,d)$  together
with a  bilinear map $ [- , - ] \colon L \times L \to L$  (called bracket)
of degree 0, such that the following conditions are satisfied:
\begin{enumerate}

\item (graded  skewsymmetry) $[a,b]=-(-1)^{\bar{a} \; \bar{b}} [b,a]$.

\item (graded Jacobi identity) 
$ [a,[b,c]] = [[a,b],c] + (-1)^{\bar{a}\;\bar{b}} [b, [a,c]]$.

\item (graded Leibniz rule) 
$ d[a,b] =[ da,b]+ (-1)^{\bar{a}}[a, db]$.

\end{enumerate}

In particular, the Leibniz rule implies that the bracket of a DGLA induces
 a structure of graded Lie algebra on its cohomology. 
 Moreover, a DGLA is \emph{abelian} if its bracket is trivial.

A \emph{morphism} of differential graded Lie algebras $\chi \colon L \to M$ is  a linear map
that commutes with brackets and  differentials and preserves degrees.

A \emph{quasi-isomorphism} of DGLAs is a morphism
that induces an isomorphism in cohomology.  Two DGLAs $L$ and $M$ are said to be
\emph{quasi-isomorphic}, or \emph{homotopy equivalent}, if they are equivalent under the
equivalence relation generated by:  $L\sim M$ if there exists
a quasi-isomorphism $\chi\colon L\to M$. 
A DGLA is  \emph{homotopy abelian} if it   is quasi-isomorphic to an abelian  DGLA.

\begin{remark}
The   category ${\bf{DGLA}}$ of DGLAs is too strict for our purpose and we require to enhance
this category allowing $L_{\infty}$ morphisms of DGLAs. Therefore, we work in the category whose objects are DGLAs and whose morphisms are  $L_{\infty}$ morphisms of DGLAs.
 This category is equivalent to the homotopy category of  ${\bf{DGLA}}$, obtained inverting all quasi-isomorphisms.
Using this fact, we do not give the explicit definition of an $L_{\infty}$ morphism of DGLAs: by an $L_{\infty}$ morphism we mean a morphism in this homotopy category (a zig-zag morphism) and we denote it with a dash-arrow. We only  emphasize that an  $L_{\infty}$ morphism of DGLAs has a linear part that is a morphism of complexes and therefore it induces a  morphism in cohomology.  
For   the detailed  descriptions of such structures we refer to
\cite{LadaStas,LadaMarkl,EDF,fukaya,K,getzler,manRENDICONTi,cone,IaconoIMNR}.

\end{remark}

\begin{lemma}\label{lem.criterioquasiabelianita}
Let $f_{\infty}\colon M_1 \dashrightarrow M_2$ be  a $L_{\infty}$ morphism of
DGLAs
with $M_2$ homotopy abelian.  If $f_{\infty}$ induces  an injective morphism in cohomology, then $M_1$ is also homotopy abelian.
\end{lemma}

\begin{proof}
See \cite[Proposition 4.11]{KKP} or \cite[Lemma 1.10]{algebraicBTT}.
\end{proof}

The \emph{homotopy fibre} of a morphism of DGLA $\chi\colon L\to M$ is the DGLA
\[TW(\chi):=\{(l, m(t,dt)) \in L \times M[t,dt] \ \mid \  m(0,0)=0, \, m(1,0)=\chi(l) \}.
\]

\begin{remark} \label{rem.quasiisoTWcono} 

If $\chi\colon L\to M$ is an injective morphism of DGLAs, 
then its cokernel $M/\chi(L)$ is a differential graded vector space and the map
\[ 
TW(\chi)\to (M/\chi(L))[-1], \qquad (l,p(t)m_0+q(t)dt m_1)
\mapsto 
\left(\int_0^1q(t)dt\right) m_1 \pmod{\chi(L)},
\]
is a surjective quasi-isomorphism.
\end{remark}

\begin{lemma} \label{lem.criterio TW abelian}
Let $\chi \colon L\to M$ be an injective  morphism of differential graded Lie algebras such that: 
$\chi \colon H^*(L)\to H^*(M)$ is injective. Then, the homotopy fibre $TW(\chi)$
is homotopy abelian.
\end{lemma}

\begin{proof}  \cite[Proposition 3.4]{algebraicBTT} or \cite[Lemma 2.1]{semireg}.

\end{proof}

\begin{example}
\cite[Example 3.5]{algebraicBTT}
 Let $W$ be a differential graded vector space and let $U \subset W$
be a differential graded subspace. If  the induced  morphism in cohomology
 $H^*(U)\to H^*(W)$ is  injective, then the inclusion  of DGLAs
\[
\chi \colon \{f\in \Hom^*_{\K}(W,W) \mid f(U) \subset U\} \to \Hom^*_{\K}(W,W)
\]
satisfies  the hypothesis of Lemma~\ref{lem.criterio TW abelian} and so the DGLA $TW(\chi)$ is   homotopy abelian.
\end{example}

\subsection{Cartan homotopies}\label{Section cartan homoto}

Let $L$ and $M$ be two differential graded Lie algebras. A \emph{Cartan homotopy} is a linear map of degree $-1$
\[ \bi \colon L \to M  \]
such that,  for every $a,b\in L$, we have:
\[ \bi_{[a,b]}=[\bi_a,d_M\bi_b] \qquad  \text{and }   \qquad [\bi_a,\bi_{b}] =0.\]

For every Cartan homotopy $\bi$, it is defined the Lie derivative  map
\[ \bl \colon L\to M,\qquad
\bl_a=d_M\bi_a+\bi_{d_L a}.
\]
It follows from the definiton of  a Cartan homotopy $\bi$ that  $\bl$ is a morphism of DGLAs. 
Therefore, the conditions    of Cartan homotopy become
\[\bi_{[a,b]}=[\bi_a,\bl_b] \qquad \text{and } \qquad [\bi_a,\bi_{b}]=0.\]
Note that, as a morphism of complexes, $\bl$ is homotopic to 0 (with homotopy $\bi$).

\begin{example}\label{exam.cartan su ogni aperto}

Let $X$ be a smooth algebraic  variety. Denote by $\Theta_X$ the tangent
sheaf and by $(\Omega^{\ast}_X,d)$ the algebraic de Rham complex.
Then, for every open subset $U \subset X$,  the contraction of a vector space with a differential form
\[ 
\Theta_X(U) \otimes \Omega^k_X(U) \xrightarrow{\quad\contr\quad}
 \Omega^{k-1}_X(U)
\]
induces a linear  map of degree $-1$
\[
\bi \colon  \Theta_X(U) \to  \Hom^*(\Omega^{*}_X(U),
\Omega^{*}_X(U)), \qquad  \bi_{\xi} (\omega) = \xi \contr\omega
\]
that  is a Cartan homotopy. Indeed, the above conditions coincide 
with the classical Cartan's homotopy formulas. 
\end{example}

We are interested in the logarithmic generalization of the previous example.

\begin{example}\label{exam.cartan relativo su ogni aperto}

Let $X$ be a smooth algebraic  variety and $D$ a  normal crossing divisor.
Let $(\Omega^{\ast}_X(\log D),d)$  be the logarithmic differential complex and $\Theta_X(-\log  D)$   the subsheaf  of the tangent sheaf  $\Theta_X$  of the derivations  that preserve the ideal sheaf of $D$ as in the previous section.
It is easy to prove explicitly that for every open subset $U\subset X$, we have
\[(\, \Theta_X(-\log  D)(U)\ \contr\ \Omega^k_X(\log D)(U)\,) \subset
\Omega^{k-1}_X(\log D)(U).\]
Then, as above, the induced linear map of degree $-1$
\[
\bi\colon \Theta_X(-\log  D)(U)\to \Hom^*(\Omega^{*}_X(\log  D)(U),
\Omega^{*}_X(\log  D)(U)),\qquad \bi_{\xi}(\omega)=\xi\contr\omega
\]
is a Cartan homotopy.
\end{example}

\begin{lemma} \label{lem.cartan induce morfismo TW}
Let $L,M$ be DGLAs and $\bi \colon L\to M$   a Cartan homotopy. Let $N\subset M$ be a differential graded Lie subalgebra such that $\bl(L)\subset N$ and
\[ TW(\chi)= \{(x,y(t))\in N\times M[t,dt]\mid y(0)=0,\; y(1)=x\}\]
the homotopy fibre of the inclusion $\chi \colon N\hookrightarrow M$.  Then, it is well defined an $L_{\infty}$ morphism $L\stackrel{(\bl,\bi)}{\dashrightarrow }TW(\chi)$.
%

 \end{lemma}
 
\begin{proof}
 See \cite[Corollary 7.5]{semireg} for an explicit description of this morphism. We only note that the linear part, i.e., the induced morphism of complexes, is given by $(\bl,\bi)(a):= (\bl_a,t\bl_a+dt\bi_a)$, for any $a \in L$.
\end{proof}

\subsection{Simplicial objects and Cartan homotpies}

Let $\mathbf{\Delta}_{\operatorname{mon}}$ be  the category whose objects are finite
ordinal sets and whose morphisms are order-preserving injective
maps between them.  
A \emph{semicosimplicial differential graded Lie algebra} is a
covariant functor $\mathbf{\Delta}_{\operatorname{mon}}\to
\mathbf{DGLA}$. Equivalently, a
semicosimplicial DGLA ${\mathfrak g}^\Delta$ is a diagram
 \[
\xymatrix{ {{\mathfrak g}_0}
\ar@<2pt>[r]\ar@<-2pt>[r] & { {\mathfrak g}_1}
      \ar@<4pt>[r] \ar[r] \ar@<-4pt>[r] & { {\mathfrak g}_2}
\ar@<6pt>[r] \ar@<2pt>[r] \ar@<-2pt>[r] \ar@<-6pt>[r]&
\cdots},
\]
where each  ${\mathfrak g}_i$ is a DGLA, and for each
$ i > 0 $, there are $ i + 1$ morphisms of DGLAs
\[
\partial_{k,i} \colon  {\mathfrak g}_{i-1}\to {\mathfrak
g}_{i},
\qquad k=0,\dots,i,
\]
such that $\partial_{ k+1, i+1} \partial_{l , i}= \partial_{l,i+1}\partial_{k,i}$,
for any  $k\geq l$.

In a semicosimplicial DGLA  ${\mathfrak g}^\Delta$, the maps
\[
\partial_i=\partial_{0,i}-\partial_{1,i}+\cdots+(-1)^{i}
\partial_{i,i}
\]
endow the vector space $\prod_i{\mathfrak g}_i$ with the
structure of a differential complex. Moreover, being a
DGLA, each ${\mathfrak g}_i$ is in particular a differential
complex; since the maps $\partial_{k,i}$ are morphisms of DGLAs,
the space $
{\mathfrak g}^\bullet_\bullet$
has a natural bicomplex structure. We emphasise  that the associated total complex
\[({\rm Tot}({\mathfrak g}^\Delta),d_{\tot})\quad\text{where}\quad
{\rm Tot}({\mathfrak g}^\Delta)=\prod_{i}{\mathfrak
g}_i[-i],\quad d_{\tot}=\sum_{i,j}\partial_i+(-1)^jd_j\] has no
natural DGLA structure.
However,  there is another bicomplex  naturally associated with a semicosimplicial DGLA, whose total complex is naturally a
DGLA.

For every $n\ge 0$, let
 $(A_{PL})_n$ be the differential graded commutative algebra
of polynomial differential forms on the standard $n$-simplex
$\{(t_0,\ldots,t_n) \in \K^{n+1}\mid \sum t_i=1\}$ \cite{FHT}:
\[ (A_{PL})_n = \frac{\K[t_0,\ldots,t_n,dt_0,\ldots,dt_n]}
{(1-\sum t_i, \sum dt_i)}.\]
Denote
by $\delta^{k,n} \colon (A_{PL})_n \to (A_{PL})_{n-1}$, $k = 0,\ldots,n$,
the face maps; then, there are well-defined
morphisms of   differential graded vector spaces
\[
 \delta^{k} \otimes Id \colon  (A_{PL})_{n} \otimes  \mathfrak{g}_n  \to (A_{PL})_{n-1} \otimes
 \mathfrak{g}_n,\]
\[Id \otimes \partial_{k} \colon
(A_{PL})_{n-1}   \otimes  \mathfrak{g}_{n-1}   \to (A_{PL})_{n-1} \otimes  \mathfrak{g}_{n},
\]
for every $0\le k\le n$.
The Thom-Whitney bicomplex is then defined as
\[
C^{i,j}_{TW}(\mathfrak{g}^\Delta)
=\{ (x_n)_{n\in {\mathbb
N}}\in \prod_n (A_{PL})_n^i\otimes {\mathfrak g}_n^j
\mid ( \delta^{k} \otimes Id)x_n=
(Id \otimes \partial_{k})x_{n-1},\; \forall\; 0\le k\le n\},
\]
where  $(A_{PL})_n^i$  denotes the degree $i$ component of $(A_{PL})_n$.
Its total complex is denoted by $( {TW}(\mathfrak{g}^\Delta), d_{TW})$ and it 
is a DGLA, called the  \emph{Thom-Whitney} DGLA.
Note  that the integration maps
\[ \int_{\Delta^n}\otimes \operatorname{Id}\colon (A_{PL})_{n}\otimes
 {\mathfrak g}_n\to {\mathbb K}[n]\otimes  {\mathfrak g}_n= {\mathfrak g}_n[n]\]
 give a quasi-isomorphism
of differential graded vector spaces
\[
I\colon ( TW( {\mathfrak g}^\Delta), d_{TW})\to
({\tot}( {\mathfrak g}^\Delta),d_{\tot}).
\]
For more details, we refer the reader to \cite{whitney,navarro,getzler,cone,chenggetzler}.

\begin{remark}
For any  semicosimplicial DGLA ${\mathfrak g}^\Delta$, we have just defined the Thom-Whitney DGLA. Therefore, using the Maurer-Cartan equation, 
 we can associate with any ${\mathfrak g}^\Delta$ a deformation functor,   namely 
\[
\Def_{{TW}(\g^{\Delta})}: \Art \to \Set,
\]
\[
\Def_{{TW}(\g^{\Delta})}(A)=\frac{\MC_{{TW}(\g^{\Delta})}(A)}{\text{gauge}}=\frac{\{ x \in {{TW}(\g^{\Delta})}^1\otimes \mathfrak{m}_A \ |\ dx+
\displaystyle\frac{1}{2} [x,x]=0 \}}{\exp({{TW}(\g^{\Delta})}^0\otimes \mathfrak{m}_A )  }.
\]
In particular,   the tangent space to $\Def_{{TW}(\g^{\Delta})}$ is
\[
T\Def_{{TW}(\g^{\Delta})}:=\Def_{{TW}(\g^{\Delta})}( \K[\epsilon]/  \epsilon^2 )  \cong H^1({TW}(\g^{\Delta}))\cong H^1({\tot}(\g^{\Delta}))
\]
and obstructions are contained in
\[
H^2({TW}(\g^{\Delta}))\cong H^2({\tot}(\g^{\Delta})).
\]
\end{remark}

\begin{example}\label{ex.cech semicosimplicial}
Let $\sL$ be a  sheaf of differential graded vector spaces over an algebraic variety $X$ and  $\mathcal{U}=\{U_i\}$  an  open cover of $X$; assume that the set of indices $i$ is totally ordered. We can then  define
the semicosimplicial DG vector space of  \v{C}ech  cochains of $\sL$ with respect to the cover $\mathcal{U}$:
\[ \sL(\mathcal{U}):\quad \xymatrix{ {\prod_i\mathcal{L}(U_i)}
\ar@<2pt>[r]\ar@<-2pt>[r] & { \prod_{i<j}\mathcal{L}(U_{ij})}
      \ar@<4pt>[r] \ar[r] \ar@<-4pt>[r] &
      {\prod_{i<j<k}\mathcal{L}(U_{ijk})}
\ar@<6pt>[r] \ar@<2pt>[r] \ar@<-2pt>[r] \ar@<-6pt>[r]& \cdots},\]
where the coface maps   $  \displaystyle \partial_{h}\colon
{\prod_{i_0<\cdots <i_{k-1}} \sL(U_{i_0 \cdots  i_{k-1}})}\to
{\prod_{i_0<\cdots <i_k} \sL(U_{i_0 \cdots  i_k})}$
are given by
\[\partial_{h}(x)_{i_0 \ldots i_{k}}={x_{i_0 \ldots
\widehat{i_h} \ldots i_{k}}}_{|U_{i_0 \cdots  i_k}},\qquad
\text{for }h=0,\ldots, k.\]

The total complex $\tot(\sL(\mathcal{U}))$ is the associated \v{C}ech complex $C^*(\mathcal{U},\sL)$ and we denote by $TW(\sL(\mathcal{U}))$ the associated Thom-Whitney complex. The integration map
$TW(\sL(\mathcal{U}))\to C^*(\mathcal{U},\sL)$ is a surjective quasi-isomorphism.
If $\sL$ is a quasicoherent DG-sheaf and every $U_i$ is affine, then the cohomology of  $TW(\sL(\mathcal{U}))$ is the same of the cohomology of $\sL$.

\end{example}

\begin{example}\label{example.funtore se DGLA sono LIE} \cite{FIM,FMM}
If each $\g_i$ is concentrated in degree zero, i.e.,
$\g^\Delta$ is a semicosimplicial  Lie
algebra,  then the functor $\Def_{ TW(\g^{\Delta})} $ has another explicit description; namely, it is isomorphic to the following functor:
\[
H^1_{\rm sc}(\exp \g^\Delta): \Art \to \Set
\]
\[
H ^1_{sc}(\exp { \g}^\Delta )(A)=\frac{\{ x \in {\g}_1 \otimes
\mathfrak{m}_A   \ |\  e^{\de_{0}x}e^{-\de_{1}x}e^{\de_{2}x}=1
\}}{\sim},
\]
where $x \sim y$ if and only if there exists
$a\in {\g}_0\otimes\mathfrak{m}_A$, such that
$e^{-\de_{1}a}e^{x}e^{\de_{0}a}=e^y$.

\medskip

In particular, let  $Z \subset X$ be  a closed subscheme of a smooth variety $X$, $\mathcal{U}=\{U_i\}$   an open affine cover of $X$ and  consider 
 $\g^\Delta = {TW( {\Theta}_X(-\log Z)(\mathcal{U}))}$. Then, for every  $A \in \Art$, we have
\[
\Def_{TW(\g^\Delta  )}(A)\cong\frac{\{ \{x_{ij}\} \in \prod_{i<j} {\Theta}_X(-\log Z)(U_{ij})
 \otimes
\mathfrak{m}_A   \ |\  e^{x_{jk}} e^{-x_{ik}} e^{x_{ij}}=1
\}}{\sim},
\]
where $x \sim y$ if and only if there exists
$\{a_i\}_i \in \prod_i{\Theta}_X(-\log Z)(U_{i })\otimes\mathfrak{m}_A$, such that $e^{-a_i}e^{x_{ij}}e^{a_j}=e^{y_{ij}}$ \cite[Theorem 4.1]{FMM}.

\end{example}

\bigskip

The notion of Cartan homotopy is related to the notion of calculus and it can be extended to the semicosimplicial setting.

\begin{definition}\label{def.contraction} \cite{TT05,semireg}
Let $L$ be a differential graded Lie algebra and $V$ a differential
 graded vector space. A bilinear map
\[ 
L \times V \xrightarrow{\quad\contr\quad} 
V
\]
of degree $-1$ is called a \emph{calculus} if the induced map
\[ \bi \colon L \to \Hom^*_{\K}(V,V), \qquad  \bi_l(v) = l\contr v,\]
is a Cartan homotopy.

\end{definition}

\begin{definition}\label{def.contractioncosimpl}
Let $\mathfrak{g}^\Delta$ be a semicosimplicial DGLA and
$V^\Delta$ a semicosimplicial differential graded vector space.
A \emph{semicosimplicial Lie-calculus}
\[ \mathfrak{g}^\Delta\times V^\Delta\xrightarrow{\;\contr\;} V^\Delta,\]
is a sequence of calculi $\mathfrak{g}_n\times V_n\xrightarrow{\;
 \contr\;} V_n$, $n\ge 0$,
commuting with coface maps, i.e., $\de_k(l\contr v)=\de_k(l)\contr
\de_k(v)$, for every $k$.
\end{definition}

\begin{lemma}\label{lem.TWforcontractions}
Every semicosimplicial calculus
\[ \mathfrak{g}^\Delta\times V^\Delta\xrightarrow{\;\contr\;} V^\Delta\]
extends  naturally to a calculus
\[  {TW}(\mathfrak{g}^\Delta)\times  {TW}(V^\Delta)
\xrightarrow{\;\contr\;}  {TW}(V^\Delta).\]
Therefore, the induced map
\[ \bi\colon  {TW}(\mathfrak{g}^\Delta)
 \to \Hom^*_{\K}({TW}(V^\Delta),{TW}(V^\Delta)) \]
is a Cartan homotopy.
\end{lemma}

\begin{proof}
 \cite[Proposition 4.9]{algebraicBTT}.
\end{proof}

\begin{example}\label{exe. cartan homoto relativa TW}

Let $X$ be a smooth algebraic  variety and $D$ a normal crossing  divisor.
 Denote by $(\Omega^{\ast}_X(\log D),d)$  the logarithmic differential complex and $\Theta_X(-\log  D)$ the usual subsheaf of $\Theta_X$ preserving the ideal of $D$.

According to Example~\ref{exam.cartan relativo su ogni aperto},
for every open subset $U\subset X$, we have a contraction
\[\Theta_X(-\log  D)(U)\ \times\ \Omega^*_X(\log D)(U)\, \xrightarrow{\;\contr\;}
\Omega^*_X(\log D)(U).\]
Since it  commutes with   restrictions to open subsets, for every affine open cover $\mathcal{U}$ of $X$, we have a
semicosimplicial contraction
\[
\Theta_X(-\log D)(\mathcal{U})\times\Omega^*_X (\log D)(\mathcal{U})
\xrightarrow{\;\contr\;} \Omega^*_X (\log D)(\mathcal{U}).
\]
According to Lemma \ref{lem.TWforcontractions}, it is well defined
the Cartan homotopy
\[\bi\colon  {TW}( \Theta_X(-\log D)(\mathcal{U})) \longrightarrow
\Hom^*( {TW}(\Omega^*_X (\log D)(\mathcal{U})),
 {TW}(\Omega^*_X (\log D)(\mathcal{U}))).\]

\end{example}

\section{Deformations of pairs}\label{section deformation}

Let $Z \subset X$ be a closed subscheme of a smooth variety $X$ and denote by $j:Z \hookrightarrow X$ the closed embedding. Note that at this point we are not assuming neither $Z$ divisor  nor $Z$ smooth.
We recall the definition of infinitesimal deformations of the closed embedding $j:Z \hookrightarrow X$, i.e., infinitesimal deformations of the pair $(X,Z)$; full details can be found for instance in \cite[Section 3.4.4]{Sernesi} or \cite{Kaw0}.
\begin{definition}
Let $A\in \Art$. An \emph{infinitesimal deformation} of  $j:Z \hookrightarrow X$
over $\Spec(A)$ is a   diagram
\begin{center}
$\xymatrix{\mathcal{Z} \ar[rr]^J\ar[dr]_p &  &\mathcal{X}   \ar[dl]^\pi  \\
          &   \Spec(A), &  &  \\ }$
\end{center}
where $p$ and $\pi$ are    flat maps, such   that the diagram is isomorphic to
$j:Z \hookrightarrow X$  via the pullback $\Spec(\K) \to \Spec(A)$. Note that $J$ is also a closed embedding \cite[pag 185]{Sernesi}.

Given another infinitesimal deformation of $j$:
\begin{center}
$\xymatrix{\mathcal{Z}' \ar[rr]^{J'}\ar[dr]_{p'} &  & \mathcal{X}'   \ar[dl]^{\pi'}  \\
          &   \Spec(A), &  &  \\ }$
\end{center}
an isomorphism between these two deformations is a pair of isomorphisms of deformations:
\[
\alpha: \mathcal{Z} \to \mathcal{Z}' , \qquad \beta: \mathcal{X} \to \mathcal{X}'
\]
such that the following diagram
\begin{center}
$\xymatrix{\mathcal{Z}  \ar[rr]^{J}\ar[d]_{\alpha} &  & \mathcal{X } \ar[d]^{\beta}  \\
         \mathcal{Z}'  \ar[rr]^{J'} &  & \mathcal{X}'   \\ }$
\end{center}
is commutative.
The associated  infinitesimal deformation functor is
$$
\Def_{(X,Z)} : \Art \to \Set,
$$
$$
\Def_{(X,Z)}(A)= \{
 \mbox{isomorphism classes of infinitesimal deformations of $j$ over
 $\Spec(A)$} \}.
$$
\end{definition}

Furthermore, we define the sub-functor
$$
{\Def'}_{(X,Z)} : \Art \to \Set,
$$
$$
{\Def'}_{(X,Z)} =
\left\{\begin{array}{c}  \mbox{  isomorphism   classes   of   locally trivial}\\
 \mbox{  infinitesimal   deformations $j$ over  $\Spec(A)$}\end{array} \right\}.
$$

\begin{remark}
Since every affine non singular algebraic variety is rigid \cite[Theorem 1.2.4]{Sernesi}, whenever $Z\subset X$ is smooth,  every deformation of $j$ is locally trivial and so ${\Def}_{(X,Z)}\cong {\Def'}_{(X,Z)}$.
\end{remark}

Let $\mathcal{U}=\{U_i\}_{i \in I}$ be an affine open cover of $X$ and   $TW( {\Theta}_X(-\log Z)(\mathcal{U}))$ the DGLA associated with the sheaf of Lie algebras  ${\Theta}_X(-\log Z)$ as in Example \ref{ex.cech semicosimplicial}.

\begin{theorem}\label{teo. DGLA controlling def of j}
In the assumption above, the DGLA $TW( {\Theta}_X(-\log Z)(\mathcal{U}))$ controls the locally trivial deformation of the closed embedding $j:Z \hookrightarrow X$, i.e., there exists an isomorphism  of deformation functors
\[
\Def_{TW( {\Theta}_X(-\log Z)(\mathcal{U}))} \cong
\Def'_{(X,Z)}.
\]
In particular, if $Z\subset X$ is smooth, then
$\Def_{TW( {\Theta}_X(-\log Z)(\mathcal{U}))} \cong
\Def_{(X,Z)}$.
\end{theorem}

\begin{proof}
See also \cite[Proposition 3.4.17]{Sernesi}, \cite[Theorem 4.2]{donarendiconti}.

Denote by $\mathcal{V}=\{V_i=U_i\cap Z\}_{i \in I}$ the induced affine open cover of  $Z$.
Every locally trivial deformation of $j$ is obtained by the gluing  of the trivial deformations
\[
\xymatrix{ V_i \ar[rr] \ar[d]& & V_{i } \times \Spec(A) \ar[d]\\
  U_i \ar[rr]& & U_{i } \times \Spec(A), \\ }
\]
in a compatible way along the double intersection $V_{ij} \times \Spec(A)$ and $U_{ij} \times \Spec(A)$.
Therefore, it is determined by automorphisms of the trivial deformations, that glues over triple intersections,  i.e., by pairs of automorphisms $(\varphi_{ij}, \phi_{ij})$, where
\[ \varphi_{ij} \colon V_{ij} \times \Spec(A) \to V_{ij} \times \Spec(A)
\qquad  \mbox{ and } \qquad \phi_{ij} \colon U_{ij} \times \Spec(A) \to U_{ij} \times \Spec(A)
\]
are automorphisms of deformations, satisfying the cocycle condition on triple intersection and such that the following diagram

\[
\xymatrix{ V_{ij} \times \Spec(A)  \ar[rr]^{\varphi_{ij}} \ar[d]& & V_{ij} \times \Spec(A) \ar[d]\\
  U_{ij}  \times \Spec(A)  \ar[rr]^{\phi_{ij}}& & U_{ij} \times \Spec(A) \\ }
\]
commutes.
Equivalently, we have ${\phi_{ij}}_{| V_{ij}}=\varphi_{ij}$.
Since we are in characteristic zero, we  can take the logarithms
so that  $\varphi_{ij}=e^{d_{ij}}$,  for some
$d_{ij}\in\Theta_Z(V_{ij})\otimes\mathfrak{m}_A$, and $\phi_{ij}=e^{D_{ij}}$,  for some $D_{ij}\in\Theta_X(U_{ij})\otimes\mathfrak{m}_A$. The compatibility  condition is equivalent to the condition $D_{ij} \in \Gamma(U_{ij},{\Theta}_X(-\log Z))\otimes\mathfrak{m}_A$. Summing up, a deformation of $j$ over $\Spec (A)$ corresponds to the datum of a sequence $\{D_{ij}\}_{ij} \in \prod_{ij}{\Theta}_X(-\log Z))(U_{ij})\otimes\mathfrak{m}_A$ satisfying the cocycle condition
\begin{equation}\label{equa.cociclo auto Vij}
e^{D_{jk}} e^{-D_{ik}} e^{D_{ij}}  =\Id , \qquad
   \forall \ i<j<k \in I.
\end{equation}

Next, let $J'$ be another deformation of
$j$ over  $\Spec(A)$. To give an isomorphism of deformations between $J$ and $J'$ is equivalent to give, for every $i$, an automorphism $\alpha_i$ of $V_{i } \times \Spec(A)$ and an automorphism  $\beta_i$ of $U_{i } \times \Spec(A)$, that are isomorphisms of deformations of $X$ and $Z$, respectively, i.e.,  for every $i<j$,
$\varphi_{ij}= {\alpha_i}^{-1}
{\varphi_{ij}'}^{-1}\alpha_j$ and  $\phi_{ij}= {\beta_i}^{-1}
{\phi_{ij}'}^{-1}\beta_j$.
Moreover, they have to be compatible, i.e.,   the following diagram
\begin{center}
$\xymatrix{V_i \times \Spec(A) \ar[rr] \ar[d]_{\alpha_i} &  & U_i \times \Spec(A) \ar[d]^{\beta_i}  \\
        V_i \times \Spec(A)  \ar[rr]  &  & U_i \times \Spec(A)    \\ }$
\end{center}
has to commutes, for every $i$.

Taking again logarithms, an isomorphism between the deformations $J$  and $J'$  is equivalent to the existence of a sequence   $\{a_i\}_i \in \prod_i\Theta_X(-\log  Z)(U_i)\otimes\mathfrak{m}_A$,
 such that $e^{-a_i}e^{D'_{ij}}e^{a_j}=e^{D_{ij}}$. Then, the conclusion follows from
the explicit description of the functor $
\Def_{TW( {\Theta}_X(-\log Z)(\mathcal{U}))} $ given in 
  Example \ref{example.funtore se DGLA sono LIE}.
\end{proof}

\begin{remark}\label{remark def X as def of trivial pair}
If $Z=0$, then we are analysing nothing more than the infinitesimal deformations of  the smooth variety $X$ and they  are controlled by the tangent sheaf, i.e., 
$ \Def_{TW( {\Theta}_X (\mathcal{U}))} \cong \Def_X$, for any open affine cover $\mathcal{U}$ of $X$ 
\cite[Theorem 5.3]{algebraicBTT}.
\end{remark}

\begin{example}\label{exa DGLA (X,D) on C}
In the case $\K=\C$, we can consider the DGLA $(A^{0,*}_X({\Theta}_X(-\log Z))= \oplus_i \Gamma(X, \sA^{0,i}_X({\Theta}_X(-\log Z))), \debar, [ , ])$  as  an explicit model for $TW( {\Theta}_X(-\log Z)(\mathcal{U}))$ \cite[Section 5]{ManettiSemireg} \cite[Corollary V.4.1]{Iaconophd}.
\end{example}

\section{Obstructions of pairs}
\label{section obstruction computations}
  
In this section, we analyse obstructions for the infinitesimal deformations of pairs, whenever  the sub variety is a divisor, so that we can make use of the   logarithmic differential complex  $(\Omega^{\ast}_X(\log D),d)$.

\begin{theorem}\label{thm.maintheorem}
Let  $X$ be a smooth projective variety  of dimension $n$, defined
over an algebraically closed field of characteristic 0 and  $D \subset X$ a smooth divisor. If
the contraction map
\begin{equation}\label{eqazione semiregolarity}
H^*(X,\Theta_X (-\log D))\xrightarrow{\bi}\Hom^*(H^*(X,\Omega^n_X (\log D)),H^*(X,\Omega^{n-1}_X  (\log D))) 
\end{equation}
is injective, then  the DGLA $TW( {\Theta}_X(-\log D)(\mathcal{U}))$ is homotopy abelian, for every affine open cover $\mathcal{U}$ of
$X$.
\end{theorem}

\begin{proof}
According to  Lemma~\ref{lem.criterioquasiabelianita}, it is
sufficient to prove the existence of a homotopy abelian
DGLA $H$ and an $L_\infty$-morphism
$TW( {\Theta}_X(-\log D)(\mathcal{U}))
 \dashrightarrow H$, such that the induced map of complexes is injective in cohomology.
We use the Cartan homotopy to construct the morphism, as in  Lemma \ref{lem.cartan induce morfismo TW}, and the homotopy fibre construction to provide an homotopy abelian DGLA $H$, as in Lemma \ref{lem.criterio TW abelian}.

Let $\mathcal{U}$ be an affine open  cover of $X$. For every $i\le n$, denote by $\check{C}(\mathcal{U},\Omega^i_X (\log D))$   the \v{C}ech complex of the coherent sheaf $ \Omega^i_X (\log D)$,  and
$\check{C}(\mathcal{U},\Omega^\ast_X (\log D))$  the total complex of the
logarithmic de Rham complex $\Omega^*_X (\log D)$ with
respect to the cover $\mathcal{U}$, as in  Example \ref {ex.cech semicosimplicial}.
We note that
\[
\check{C}(\mathcal{U},\Omega^n_X (\log D))^i=
\bigoplus_{a+b=i}\check{C}(\mathcal{U},\Omega^a_X (\log D))^b.
\]
and that $\check{C} (\mathcal{U}, \Omega^n_X (\log D))$ is a subcomplex of
$\check{C} (\mathcal{U}, \Omega^\ast_X (\log D))$.

\smallskip

We also have a commutative diagram of complexes  with horizontal
quasi-isomorphisms:

\[
\xymatrix{ 
\check{C}(\mathcal{U}, \Omega^n_X (\log D))
\ar[rr]  \ar@{^{(}->}[d]  &   &   {TW} ( \Omega^n_X  (\log D)(\mathcal{U}))
 \ar@{^{(}->}[d] \\  
 \check{C} (\mathcal{U},\Omega^\ast_X (\log D))  
\ar[rr]  &        &   {TW}(\Omega^\ast_X   (\log D) (\mathcal{U})).
}\]

According to Theorem \ref{teo degen deligne},  the spectral sequence  associated with the Hodge filtration  degenerates at the $E_1$-level, where  $E_1^{p,q}=H^q(X,\Omega^{p}_X(\log D))$; 
this implies that we have  the following injections:
\[ H^*(X,\Omega^n_X (\log D)) = H^*(\check{C} (\mathcal{U},\Omega^n_X (\log D)))
\hookrightarrow 
H^*(\check{C} (\mathcal{U},\Omega^*_X (\log D))).\]
\[ H^*(X,   \Omega^{n-1}_X (\log D))   =    H^*   (     \check{C}   (   \mathcal{U},  \Omega^{n-1}_X (\log D) ) )
\hookrightarrow    H^*\left  (  \frac{  \check{C}(\mathcal{U},\Omega^*_X (\log D))}{
\check{C}(\mathcal{U},\Omega^n_X (\log D))}\right).\]

Thus,  the natural inclusions of complexes
\[
{TW}(\Omega^n_X (\log D)  (\mathcal{U}))  \to {TW}(
  \Omega^\ast_X (\log D)    (\mathcal{U})),
  \]
\[ 
{TW}  (\Omega^{n-1}_X (\log D)(\mathcal{U}))\to \frac{{TW}(
\Omega^\ast_X (\log D)(\mathcal{U}))}{{TW}(\Omega^n_X (\log D)(\mathcal{U}))},\]
induces injective morphisms  in cohomology.

Consider the  differential graded Lie algebras
\[
M=\Hom^*( {TW}( \Omega^*_X (\log D)(\mathcal{U})),  {TW}
(\Omega^*_X (\log D)(\mathcal{U})),
\]
\[
L=\{f\in M \mid
f( {TW}(\Omega^n_X (\log D)(\mathcal{U})))\subset
 {TW}(\Omega^n_X (\log D)(\mathcal{U}))\},
\]
and denote by $\chi\colon L \to M$ the inclusion.  Lemma \ref{lem.criterio TW abelian} implies that the homotopy fibre $TW(\chi^{\Delta})$ is homotopy abelian.
Next, we  provide the existence of a morpshim to this homotopy abelian DGLA.

According to  Example \ref{exe. cartan homoto relativa TW}, it is well defined the Cartan homotopy
\[\bi\colon  {TW}( \Theta_X(-\log D)(\mathcal{U})) \longrightarrow
\Hom^*( {TW}(\Omega^*_X (\log D)(\mathcal{U})),
 {TW}(\Omega^*_X (\log D)(\mathcal{U}))).\]

In particular,    for every $\xi\in  {TW}(\Theta_X(-\log D)(\mathcal{U}))$
and every $k$, we note that
\[ \bi_{\xi}( {TW}(\Omega^k_X (\log D)(\mathcal{U})))\subset
  {TW}(\Omega^{k-1}_X (\log D)(\mathcal{U})),\]
\[ \bl_{\xi}( {TW}(\Omega^k_X (\log D)(\mathcal{U})))\subset
 {TW}(\Omega^k_X (\log D)(\mathcal{U})),\qquad \bl_{\xi}
=d\bi_{\xi}+\bi_{d\xi}.\]

Therefore,  $\bl( {TW}(\Theta_X(-\log D)(\mathcal{U})))\subset L$ and so, by
Lemma \ref{lem.cartan induce morfismo TW}, there exists an
$L_\infty$-morphism
\[
 {TW}( \Theta_X(-\log D)(\mathcal{U})) \stackrel{(\bl,\bi)}{ \dashrightarrow}
TW(\chi^{\Delta}).
\]
Finally, since the map $\chi$ in injective, according to Remark \ref{rem.quasiisoTWcono}, the homotopy fibre
  $TW(\chi^{\Delta})$ is quasi-isomorphic to the suspension of its  cokernel
\[\coker \chi[-1]= \Hom^*\left( {TW}(\Omega^n_X (\log D)\mathcal{U})),
\frac{ {TW}(\Omega^*_X (\log D)(\mathcal{U}))}{
 {TW}(\Omega^{n}_X (\log D)(\mathcal{U}))}\right)[-1].\]

Summing up, since the $L_\infty$-morphism induces a morphism of complexes,  we  have the following  commutative diagram of complexes
\[\xymatrix{  {TW}(\Theta_X(-\log D)(\mathcal{U}))
\ar[rr]^{(\bl,\bi)} \ar[d]^{\bi} & &  TW(\chi^{\Delta})
 \ar[d]^{q-iso} \\
\Hom^*\left( {TW}(\Omega^n_X (\log D)(\mathcal{U})),
{TW}(\Omega^{n-1}_X (\log D)(\mathcal{U}))\right)
\ar[rr]^{\qquad\qquad\qquad\alpha\!\!\!\!} & & \coker \chi[-1].
\\ }\]

By the assumption of the theorem, together with
\cite[3.1]{navarro}, the left-hand map is  injective in
cohomology.
Since  $\alpha$ is also  injective in cohomology, we conclude that
the $L_\infty$-morphism $(\bl,\bi)$ is injective in cohomology.

\end{proof}

\begin{theorem}\label{thm.kodairaprinciple}
Let  $X$  be a smooth  projective variety defined over an algebraically closed field of characteristic 0 and  $D \subset X$ a  smooth divisor.  Then, the obstructions to the deformations of the pair $(X,D)$ are contained in the kernel of  the contraction map
\[ H^2(\Theta_X( (-\log D)))\xrightarrow{\bi}\prod_{p}
\Hom(H^p((\Omega^{n}_X  (\log D)),H^{p+2}((\Omega^{n-1}_X  (\log D)).\]
\end{theorem}

\begin{proof}
Following the proof of Theorem~\ref{thm.maintheorem},
 for every affine open cover $\mathcal{U}$ of $X$,
there exists  an   $L_{\infty}$-morphism
$TW(\Theta_X( (-\log D)(\mathcal{U})))\dashrightarrow TW
(\chi^{\Delta})$ such that
$TW(\chi^{\Delta})$ is homotopy abelian. Therefore,  the deformation functor
associated with $TW(\chi^{\Delta})$ is unobstructed
and the obstructions of
$\Def_{(X,D)}\simeq\Def_{TW(\Theta_X( (-\log D)(\mathcal{U}))}$ are
contained in the kernel of the obstruction map
$H^2(TW(\Theta_X( (-\log D)(\mathcal{U})))\to H^2(TW(\chi^{\Delta}))$.
\end{proof}

\begin{remark}
In the previous theorem, we prove that all obstructions are annihilated by the contraction map; in general, the $T^1$-lifting  theorem is definitely insufficient for
proving this kind of theorem, see also   \cite{IaconoSemireg,ManettiSeattle}.
\end{remark}

\begin{corollary} \label{corollari log calabi yau formal smooth}
Let  $\mathcal{U}=\{U_i\}$ be  an affine open  cover of a smooth
projective variety  $X$ defined over an algebraically closed field
of characteristic 0 and  $D \subset X$ a  smooth divisor. If $(X,D)$ is a log Calabi-Yau pair,  then  the DGLA $TW( {\Theta}_X(-\log D))(\mathcal{U})$ is homotopy abelian.
\end{corollary}

\begin{proof}
Let $n$ be the dimension of $X$, then by definition   the sheaf $\Omega^n_X (\log D)$ is trivial (Definition \ref{definiiton log calabi yau}).
 Therefore, the cup product with a nontrivial section of it  gives the isomorphisms
$H^i(X,\Theta_X (-\log D))\simeq H^i(X,\Omega^{n-1}_X  (\log D))$. Then, the conclusion follows  from Theorem~\ref{thm.maintheorem}.
\end{proof}

\begin{corollary}\label{cor.log calabi yau no obstruction}
Let  $X$  be a smooth  projective $n$-dimensional variety defined over an algebraically closed field of characteristic 0 and  $D \subset X$ a smooth divisor. If $(X,D)$ is a log Calabi-Yau pair, i.e., the logarithmic canonical bundle $\Omega^n_X (\log D)\cong \Oh(K_X+D)$ is trivial,  then the pair  $(X,D)$ has unobstructed deformations.
\end{corollary}

\begin{proof}
According to Theorem  \ref{teo. DGLA controlling def of j}, for every   affine open cover $\mathcal{U}$ of $X$, there exists an isomorphism of functor
$\Def_{(X,D)} \cong \Def_{TW((\Theta_X( (-\log D))(\mathcal{U}))}$. Then, 
  Corollary~\ref{corollari log calabi yau formal smooth} implies that they are both smooth.
\end{proof}

\begin{remark}
For the degeneration of the spectral sequence associated with the logarithmic complex, it is enough to have a normal crossing divisor $D$ in a smooth proper variety $X$ (Theorem \ref{teo degen deligne}). Therefore, we can still perform the same computations of Theorem \ref{thm.maintheorem} and prove that the obstructions to the locally trivial deformations  of a pair $(X,D)$, with $X$  smooth proper variety  and D  normal crossing divisor,  are contained in the kernel of the contraction map (\ref{eqazione semiregolarity}).
Analogously, if the sheaf $\Omega^n_X (\log D)$ is trivial,  the above computations prove the unobstructedness for  the locally trivial deformations of the pair $(X,D)$.
\end{remark}

We end this section,  by proving that  the DGLA associated with the   infinitesimal deformations  of the pair  $(X,D)$,  with $D$ a smooth divisor  in a smooth projective Calabi Yau variety $X$ is homotopy abelian; hence,  we show that the deformations  of these pairs  $(X,D)$ are unobstructed.

\begin{theorem}\label{theorem smoothness of  D inside CY}
Let  $\mathcal{U}=\{U_i\}$ be  an affine open  cover of a smooth
projective variety  $X$ of dimension n defined over an algebraically closed field
of characteristic 0 and  $D \subset X$ a  smooth divisor. 
If $\Omega ^n_X$ is trivial, i.e., $X$ is Calabi Yau, then the DGLA 
  $TW( {\Theta}_X(-\log D))(\mathcal{U})$ is homotopy abelian.
 \end{theorem}
  
\begin{proof}

The proof is similar to the one of Theorem \ref{thm.maintheorem}.
According to Theorem \ref{teo degen tensor}, the  Hodge-to-de Rham  spectral sequences associated with the complex $\Omega^*_X (\log D) \otimes \Oh_X(-D) $ degenerates at the $E_1$ level. 
Therefore,  we have  injective maps
\[  H^*(\check{C}(\mathcal{U},\Omega^n_X (\log D) \otimes \Oh_X(-D)))
\hookrightarrow
H^*(\check{C}(\mathcal{U},\Omega^*_X (\log D) \otimes \Oh_X(-D))) \]
\[ H^*(\check{C}(\mathcal{U},\Omega^{n-1}_X (\log D) \otimes \Oh_X(-D)))
\hookrightarrow
H^*\left(\frac{\check{C}(\mathcal{U},\Omega^*_X (\log D) \otimes \Oh_X(-D))}{
\check{C}(\mathcal{U},\Omega^n_X (\log D) \otimes \Oh_X(-D))}\right).\]
and so the natural inclusions of complexes
\begin{equation}\label{eq tot injective in cohomology Omega tensor}
 {TW}(\Omega^n_X (\log D) \otimes \Oh_X(-D)(\mathcal{U}))\to {TW}(
\Omega^*_X (\log D) \otimes \Oh_X(-D)(\mathcal{U})),
\end{equation}
\begin{equation}\label{eq tot injective in cohomology Omega  n-1 tensor in quot}
 {TW}(\Omega^{n-1}_X (\log D) \otimes \Oh_X(-D)(\mathcal{U}))\to \frac{ {TW}(
\Omega^*_X (\log D) \otimes \Oh_X(-D)(\mathcal{U}))}{ {TW}(\Omega^n_X (\log D) \otimes \Oh_X(-D)(\mathcal{U}))},
\end{equation}
are  injective in cohomology. Observe that, $\Omega^n_X (\log D)\otimes \Oh_X(-D) \cong \Omega^n_X$,
 according to Remark \ref{remark exac sequence Omega(logD)(-D)}.
Consider the inclusion of DGLAs  $\chi:L \to M$, where 
 \[
M=\Hom^*( {TW}( \Omega^*_X (\log D)\otimes \Oh_X(-D) (\mathcal{U})),  {TW}
(\Omega^*_X (\log D)\otimes \Oh_X(-D)  (\mathcal{U})))
\]
and
\[
L=\{f\in M \mid
f( {TW}( \Omega^n_X  (\mathcal{U})))\subset
 {TW}( \Omega^n_X (\mathcal{U}))\}.
\]
is the   sub-DGLA preserving $  {TW}( \Omega^n_X (\mathcal{U}))={TW}(\Omega^n_X (\log D)\otimes \Oh_X(-D)  (\mathcal{U}))$.

Note that
\[\coker  \chi[-1]= \Hom^*\left( {TW}(\Omega^n_X( \mathcal{U})),
\frac{ {TW}(\Omega^*_X (\log D)\otimes \Oh_X(-D)(\mathcal{U}))}{
 {TW}(\Omega^{n}_X  (\mathcal{U}))}\right)[-1].
 \]
Lemma \ref{lem.criterio TW abelian}  together with Equation \eqref{eq tot injective in cohomology Omega tensor} implies that  $TW(\chi^{\Delta})$ is homotopy abelian.

As in the proof of Theorem \ref{thm.maintheorem},  it is well defined the Cartan homotopy
\[\bi\colon  {TW}( \Theta_X(-\log D)(\mathcal{U})) \longrightarrow M
 \]
and, in particular,  $\bl( {TW}(\Theta_X(-\log D)(\mathcal{U})))\subset L$. Therefore, 
Lemma \ref{lem.cartan induce morfismo TW} implies the existence of an
$L_\infty$-morphism
\[
 {TW}( \Theta_X(-\log D)(\mathcal{U})) \stackrel{(\bl,\bi)}{ \dashrightarrow}
TW(\chi^{\Delta}).
\]

According to Lemma~\ref{lem.criterioquasiabelianita}, to conclude the proof it is enough to show that this map is 
injective in cohomology.
As morphism of complexes, the previous map fits in the following  commutative  diagram of complexes
\[\xymatrix{  {TW}(\Theta_X(-\log D)(\mathcal{U}))
\ar[rr]^{(\bl,\bi)} \ar[d]^{\bi} & &  TW(\chi^{\Delta})
 \ar[d]^{q-iso} \\
\Hom^*\left( {TW}(\Omega^n_X(\mathcal{U})),
{TW}(\Omega^{n-1}_X (\log D)\otimes \Oh_X(-D) (\mathcal{U}))\right)
\ar[rr]^{\qquad\qquad\qquad\alpha\!\!\!\!} & & \coker \chi[-1],
\\ }\]
where $\alpha$ is injective in cohomology by Equation \eqref{eq tot injective in cohomology Omega  n-1 tensor in quot}.

At this point, we use the fact that $X$ is a smooth projective  Calabi Yau variety.
Since $\Omega ^n_X$ is trivial, the cup product with a nontrivial section gives the isomorphisms
$H^i(X,\Theta_X (-\log D))\simeq H^i(X,\Omega^{n-1}_X  (\log D)\otimes \Oh_X(-D) )$, for every $i$.
Therefore, the left map in the diagram is injective in cohomology and so the same holds for  ${(\bl,\bi)}$.

\end{proof}

\begin{corollary}\label{cor. D in calabi yau no obstruction}
Let  $X$  be a smooth  projective   Calabi Yau variety  defined over an algebraically closed field of characteristic 0 and  $D \subset X$ a smooth divisor. Then, the  pair  $(X,D)$ has unobstructed deformations.
\end{corollary}

\section{Application to cyclic covers}\label{section cyclic cover}

Let $X$ be a smooth projective variety over an algebraically
closed field $\K$ of characteristic 0. If $X$ has trivial canonical bundle, then the deformations of  $X$ are unobstructed. It is actually enough that the  canonical bundle $K_X$ is a torsion line  bundle, i.e., there exists $m>0$ such that $K_X^{ m}= \Oh_X$, see for instance \cite[Corollary 2]{zivran}, \cite[Corollary B]{manetti adv}, \cite[Corally 6.5]{algebraicBTT}. 
Indeed, consider the unramified  m-cyclic cover defined by the line bundle $L=K_X$, i.e., 
$\pi: Y= \Spec (\bigoplus_{i=0}^{m-1} L^{-i}) \to X $. Then, $\pi$  is a finite flat  map of degree m and  $Y$ is a smooth projective variety with trivial canonical bundle ($K_Y \cong \pi^* K_X\cong \Oh_Y$) and so it has unobstructed deformations.
Let $\mathcal{U}=\{U_i\}_i$ be an affine cover of $X$ and fix $\mathcal{V}=\{\pi^{-1}(U_i)\}_i$ the induced cover of $Y$. Then,  the pull back map induces a morphism of DGLAs
$ TW( \Theta_X(\mathcal{U})) \to TW( \Theta_Y(\mathcal{V}))  $
 that is injective in cohomology; since the DGLA $ TW( \Theta_Y(\mathcal{V})) $  is homotopy abelian, Lemma \ref{lem.criterioquasiabelianita}  implies that  $TW( \Theta_X(\mathcal{U}))$  is also homotopy abelian   and so $\Def_X$ is unobstructed \cite[Theorem 6.2]{algebraicBTT}.

As observed in Remark \ref{remark def X as def of trivial pair}, the infinitesimal deformations of $X$ can be considered as deformations of the pair $(X,D)$ with  $D=0$.
Then,  according to the Iitaka's philosophy and inspired by \cite{Sano}, the idea is to extend the previous computations  to the logarithmic case, i.e., a pair $(X,D)$ with $D$ a smooth divisor,  by considering 
 cyclic covers of $X$ branched on the divisor $D$ (indeed, if $D=0$ we obtain the unramified  covers). 
 
 We firstly recall  some properties of these covers; for full details see for instance \cite{pardini},
\cite[Section 3]{librENSview} or \cite[Section 2.4]{kollarmori}.
Suppose we have a   line bundle $L$ on $X$, a positive integer $m \geq 1$ and a non zero section $s \in \Gamma (X, L^{ m})$ which defines the smooth divisor $D \subset X$ (as usual $L^m$ stands for $L^{\otimes m}$). 
The  cyclic cover  $\pi: Y  \to X$ of degree $m$ and  branched over $D$ is, in the language of \cite{pardini}, the simple abelian cover  determined by its building data $L$ and $D$, such that $mL \equiv D$, associated with the cyclic group $G$ of order $m$. 
More explicitly, the variety $Y=\Spec (\bigoplus_{i=0}^{m-1} L^{-i}) $ is smooth and there exists 
 a section $s' \in \Gamma(Y, \pi^* L)$, with $(s')^m = \pi^* s$.  The divisor $\Delta= (s')$ is also smooth and maps isomorphically to $D$ so that $\pi^* D= m \Delta$ and $\pi^*L= \Oh_Y(\Delta)$. Moreover,
 \[
 \pi_* \Oh_Y = \bigoplus_{i=0}^{m-1} L^{-i},
\qquad  K_Y=\pi^*K_X \otimes \Oh_Y((m-1)\Delta)= \pi^*(K_X \otimes L^{m-1})
 \]
and 
 \[
 \pi^* \Omega_X^i(\log D)\cong  \Omega_Y^i(\log \Delta) \ \mbox{ for all } \ i;
 \]
in particular, $ K_Y\otimes  \Oh_Y(\Delta)=\pi^*(K_X) \otimes \Oh_Y(m\Delta)= \pi^*(K_X \otimes\Oh_X(D))$
 \cite[Lemma 3.16]{librENSview} or \cite[Proposition 4.2.4]{lazar}.

\medskip

Since $\pi:Y \to X$ is a finite map, for any sheaf $\mathcal{F}$ on $Y$, the higher  direct images  sheaves vanish and so the Leray spectral sequence  $E_2^{p,q} =H^p(X, R^q\pi_* \sF) \Rightarrow H^{p+q}(Y,\sF)$  degenerates at level $E_2$; therefore,  it induces isomorphisms:
$
H^p(X, \pi_* \sF) \cong H^{p}(Y,\sF), \ \forall \ p.
$
 In particular, for any locally free sheaf $\mathcal{E}$ on $X$ we have:
\[
H^p(X, \pi_*\pi^* \mathcal{E}) \cong H^{p}(Y, \pi^*\mathcal{E}), \quad \forall \ p.
\]
By the projection formula
\[
\pi_*\pi^* \mathcal{E} \cong \pi_*(\pi^* \mathcal{E}\otimes \Oh_Y) \cong   \mathcal{E}\otimes \pi_* \Oh_Y
\cong \mathcal{E}\otimes \bigoplus_{i=0}^{m-1} L^{-i} ;\]
then,  for any locally free sheaf $\mathcal{E}$ on $X$    
\begin{equation}\label{equation. comolo summand G action}
H^p(X, \mathcal{E}\otimes   L^{ -i} ) \subseteq  H^p(X,  \pi_*\pi^* \mathcal{E})\cong H^{p}(Y, \pi^*\mathcal{E})  , \ \forall \ p, \, i.
\end{equation}
and in particular
\begin{equation}\label{equation. comolo summand invariant G action}
H^p(X, \mathcal{E}) \subseteq  H^p(X,  \pi_*\pi^* \mathcal{E})\cong H^{p}(Y, \pi^*\mathcal{E})  , \ \forall \ p.
\end{equation}
 
\begin{remark}
Note that, the m-cyclic group $G$ acts on $\pi_*\pi^* \mathcal{E} $: the invariant summand of $\pi_*\pi^* \mathcal{E}  $ is $(\pi_*\pi^* \mathcal{E} )^{inv} =  \mathcal{E} $, while $ \mathcal{E}  \otimes   L^{ -i}$ is the direct summand of  $\pi_*\pi^* \mathcal{E} $ on which $G$ acts via multiplication by $\zeta ^i \ (\zeta^m =1)$.
\end{remark}

\begin{proposition}\label{proposition morismo DGLA cover}
In the above notation, let  $\pi: Y  \to X$ be the  $m$-cyclic cover   branched over $D$ with $\pi^* D= m \Delta$. Let $\mathcal{U}=\{U_i\}_i$ be an affine open cover of $X$ and  $\mathcal{V}=\{\pi^{-1}(U_i)\}_i$ the induced affine open cover  of $Y$; then, the pull back define a morphism of DGLAs
\[
TW( \Theta_X(-\log D) (\mathcal{U})) \to TW(  {\Theta}_Y(-\log \Delta)(\mathcal{V})) 
\]
that is injective in cohomology.

\end{proposition}

\begin{proof}
Let $U\subset X$ be an affine open subset and $ V=\pi^{-1}(U) $. Then, the pull back map induce a morphism
$ {\Theta}_X(-\log D)(U) \to \pi^*  {\Theta}_X(-\log D)(V)$, that behaves well under the restriction to open sets.
Therefore, fixing  an affine cover $\mathcal{U}=\{U_i\}_i$   of $X$ and  denoting by  $\mathcal{V}=\{\pi^{-1}(U_i)\}_i$ the induced affine cover of $Y$,   the pull back map induces  a morphism of DGLAs
\[
TW( \Theta_X(-\log D) (\mathcal{U})) \to TW(  \pi^*  {\Theta}_X(-\log D)(\mathcal{V})) .
\]
Since $ \pi^* \Omega_X^1(\log D)\cong  \Omega_Y^1(\log \Delta)$, we have $
\pi^*{\Theta}_X(-\log D)\cong  {\Theta}_Y(-\log \Delta) $. 
Moreover,  the pull back morphism
\[
{\Theta}_X(-\log D)\stackrel{\pi^*}{\to} \pi^*{\Theta}_X(-\log D) \cong  {\Theta}_Y(-\log \Delta)
\]
 induces injective morphisms on the  cohomology groups. Indeed,  
$H^i(X, {\Theta}_X(-\log D))$ is a direct summand of $H^i(X, \pi_* \pi^* {\Theta}_X(-\log D)) \cong  H^i(Y, \pi^* {\Theta}_X(-\log D)) \cong   H^i(Y,  {\Theta}_Y(-\log \Delta)) $. 
It follows that the induced DGLAs morphism
\[
TW( \Theta_X(-\log D) (\mathcal{U})) \to TW(  {\Theta}_Y(-\log \Delta)(\mathcal{V})) 
\]
is injective in cohomology.
%
%
%

\end{proof}

\begin{remark}\label{remark no obstruct up implies no obstr down}
The DGLAs morphism $TW( \Theta_X(-\log D) (\mathcal{U})) \to TW(  {\Theta}_Y(-\log \Delta)(\mathcal{V})) $, induces a morphism of the associated deformation functor
\[ \Def_{(X,D)} \to \Def_{(Y,\Delta)}.
\]
According to Lemma  \ref{lem.criterioquasiabelianita},  if $TW(  {\Theta}_Y(-\log \Delta)(\mathcal{V})) $ is homotopy abelian,  so that the deformations of the pair $(Y, \Delta)$ are unobstructed, then  $TW( \Theta_X(-\log D) (\mathcal{U})$ is also homotopy abelian and so  the deformations of the pair $(X,D) $ are also unobstructed.
In particular, this happen if the pair $(Y, \Delta)$ is a log Calabi-Yau. Note that this is a sufficient but not necessary  condition for the unosbtructedness of $(X,D) $, as we can observe in the following example.

 \end{remark}

\begin{proposition}\label{proposition D in -mKx smooth pair}

Let $X$ be a smooth projective variety and $D$ a smooth divisor such that $D \in | -mK_X|$, for some positive integer $m$. Then, the DGLA $TW( \Theta_X(-\log D) (\mathcal{U}))$ is homotopy abelian and so  the deformations of the pair $(X,D) $ are   unobstructed.
\end{proposition}
\begin{proof}
Let $n$ be the dimension of $X$ and consider the m-cyclic cover $\pi: Y\to X$ branched over $D$ defined by the line bundle  $L=K_X^{ - 1}$  together with a section  $s \in H^0(X, L ^{ m})$ defining $D$. Note that $ \Omega_X^n(\log D)\otimes   L^{ -m+1}  \cong L^{ - m}\otimes \Oh_X(D)\cong \Oh_X (mK_X+D)\cong \Oh_X$. 
Defining $\Delta$ as before,  i.e., $\pi^* D= m \Delta$, we also have 
 \[
  K_Y=\pi^*K_X \otimes \Oh_Y((m-1)\Delta)= \pi^*(K_X \otimes L^{ m-1})=\pi^*(  L^{ m-2})
 \]
and in particular, 
 \[
 K_Y\otimes  \Oh_Y(\Delta) = \pi^*(K_X \otimes\Oh_X(D))=\pi^*(  L^{ m-1}).
  \] 

According to Equations \eqref{equation. comolo summand G action} and \eqref{equation. comolo summand invariant G action}, we have the following inclusions
\[
H^p(Y,  {\Theta}_Y(-\log \Delta))  \supset H^p(X, \pi_*( {\Theta}_Y(-\log \Delta))^{inv}) \cong  H^p(X,  {\Theta}_X(-\log D)) \ \forall  \, p,  \]
\[
H^p(Y,   \Omega_Y^a(\log \Delta) )  \supset  H^p(X,    \Omega_X^a(\log D)\otimes   L^{ -i}  ) \  \forall \ p,\ a, \ i;
\]
in particular for  $a=n$, $p=0$ and $i=m-1$, we have
\[
H^0(Y,   \Omega_Y^n(\log \Delta) )  \supset  H^0(X,    \Omega_X^n(\log D)\otimes   L^{-m+1}  )  \cong   H^0(X, \Oh_X). 
\]

Then  the constant section of  $\Oh_X   $ gives a  section $\omega$ of the logarithmic canonical bundle $ \Omega_Y^n(\log \Delta)$, vanishing only on $\Delta$ (of order $m-1$). In particular, 
the cup product with $\omega \in  H^0(X,    \Omega_X^n(\log D)\otimes   L^{ -m+1}  )$, gives  isomorphisms 
$H^p(X,  {\Theta}_X(-\log D) )\cong  H^{p}(X,\Omega^{n-1}_X (\log \Delta)\otimes L^{ -m+1})$, for all $p$.

Therefore, the following  composition
\[\xymatrix{  H^p(Y,\Theta_Y (-\log \Delta)) \ar[rr] ^{\bi \qquad \qquad \qquad }\ & &  
\prod_j \Hom(H^j(Y,\Omega^n_Y (\log \Delta)),H^{j+p}(Y,\Omega^{n-1}_Y  (\log \Delta))) \ar[d]^{\contr \omega}\\
H^p(X,{\Theta}_X(-\log D) )  \ar@{^{(}->}[u]^{j}& & H^{p}(X,\Omega^{n-1}_X (\log \Delta)\otimes L^{ -m+1})  \ }\]
 is injective and in particular the composition $\bi \circ j$ is injective, for all $p$.

According to Proposition \ref{proposition morismo DGLA cover},
fixing  an affine cover $\mathcal{U}=\{U_i\}_i$   of $X$ and  denoting by  $\mathcal{V}=\{\pi^{-1}(U_i)\}_i$ the induced affine cover of $Y$,     the pull back map induces  a morphism of DGLAs
\[
TW( \Theta_X(-\log D) (\mathcal{U})) \to TW(  {\Theta}_Y(-\log \Delta)(\mathcal{V})) 
\]
that is injective in cohomology.
Finally,  as in  the proof of Theorem \ref{thm.maintheorem}  denote by  $TW(\chi ^{\Delta})$ the homotopy abelian differential graded Lie algebra associated with the inclusion $\chi:L \to M$, with \[
M=\Hom^*( {TW}( \Omega^*_Y (\log \Delta)(\mathcal{V})),  {TW}
\Omega^*_Y (\log \Delta)(\mathcal{V})),
\]
\[
L=\{f\in M \mid
f( {TW}(\Omega^n_Y (\log \Delta)(\mathcal{V})))\subset
 {TW}(\Omega^n_Y (\log \Delta)(\mathcal{V}))\}.
\]
Then,  the composition morphism 
\[
TW( \Theta_X(-\log D) (\mathcal{U})) \to TW(  {\Theta}_Y(-\log \Delta)(\mathcal{V})) 
  \dashrightarrow TW(\chi ^{\Delta}).
\]
is injective in cohomolgy and so by Lemma \ref{lem.criterioquasiabelianita}, the DGLA  $TW( \Theta_X(-\log D) (\mathcal{U})) $ is homotopy abelian.  
 
\end{proof}

\begin{remark}
In the case $m=2$, the results is a consequence of  Theorem \ref{theorem smoothness of  D inside CY}  and Remark \ref{remark no obstruct up implies no obstr down}.
  Indeed, in this case 
the canonical line bundle $K_Y $ of $Y$ is trivial, i.e., Y is a smooth  Calabi Yau variety and so the DGLA associated   with the pair $(Y, \Delta)$  is homotopy abelian.

\end{remark}

%


\begin{remark}
The previous proposition is a generalisation of \cite[Theorem 2.1]{Sano}, avoiding the assumption $H^1(X, \Oh_X)=0$.
Moreover, if $H^1(D, N_{D|X})=0$, then $D$ is stable in $X$, i.e., the forgetting morphism  $\phi: \Def_{(X,D)}\to \Def_X$ is smooth; this implies that the deformations of $X$ are unobstructed, e.g., deformations of weak Fano manifolds are unobstructed     \cite[Theorem 1.1]{Sano}.
\end{remark}


\section{Application to differential graded Batalin-Vilkovisky algebras}\label{section dbv}

If the ground field is $\C$, we already noticed in Example \ref{exa DGLA (X,D) on C}, that the differential graded Lie algebra    $(A_X^{0,*}( \Theta_X(-\log D)),  \debar, [ , ])$ controls   the deformations of the pair $(X,D)$, for $D$ a smooth divisor in a projective smooth manifold $X$.
In \cite {terilla, KKP, BraunLaza}, the authors  use the differential Batalin-Vilkovisky algebras and a degeneration property for these algebras to prove that the associated DGLA is  homotopy abelian.
Using the power of the notion of Cartan homotopy, we can give an alternative proof  of these results and so we provide   alternative proofs, over $\C$, of  Corollary \ref{corollari log calabi yau formal smooth} and Corollary \ref{cor. D in calabi yau no obstruction}.

First of all we recall the fundamental definitions  in this setting, for more details we refer the reader to \cite{bavi,getzler94,KKP}.

\begin{definition}\label{def dbv} 
Let $k$ be a fixed odd integer.
A \emph{differential Batalin-Vilkovisky algebra} (dBV for short) of degree $k$ over $\K$ is the data
$(A, d, \Delta)$, where $(A,d)$ is a differential $\Z$-graded  commutative algebra with unit $1\in A$, 
and $\Delta$ is an operator of degree $-k$, such that $\Delta^2=0$, $\Delta (1)=0$ and 
\begin{multline*}
\Delta(abc)+\Delta(a)bc+(-1)^{\bar{a}\;\bar{b}}\Delta(b)ac+(-1)^{\bar{c}(\bar{a}+\bar{b})}
\Delta(c)ab=\\
=\Delta(ab)c+(-1)^{\bar{a}(\bar{b}+\bar{c})}\Delta(bc)a+(-1)^{\bar{b}\bar{c}}\Delta(ac)b.
\end{multline*}
\end{definition}
The previous equality is often called the seven-term relation.
It is well known \cite{koszul} or  \cite[Section 4.2.2]{KKP}, that given a graded dBV algebra
$(A, d, \Delta)$ of degree $k$, it is canonically defined a  differential graded Lie algebra
$(\g,d,[-,-])$, where: 
$\g=A[k]$, $d_{\g}=-d_A$ and,
\[
[a,b]= (-1)^{p}(\Delta(ab)-\Delta(a)b)-a\Delta(b),\qquad a\in A^p.
\]
\smallskip

Next, let $(A, d, \Delta)$ be a dBV algebra and 
 $t$ a formal central variable of (even) degree $1+k$. Denote by $A[[t]]$ the graded vector space of formal power  series with coefficients in $A$ and by 
by  $A((t))=\bigcup_{p\in\Z} t^pA[[t]]$ the graded vector space of formal Laurent power series.
We have   $d(t)=\Delta(t)=0$ and    $d-t\Delta$ is a well-defined differential  on $A((t))$.

\begin{lemma}\label{lem.cartanhomotopyforDBV} 
In the above notation, the map
 \[\bi\colon \mathfrak{g} \to\Hom^*_{\K}(A((t)),A((t))),\qquad
a\longmapsto \bi_a(b)=\frac{1}{t}ab \]
is a Cartan homotopy.
\end{lemma}

\begin{proof}
We have to verify the two conditions of being a Cartan homotopy, given in 
Section \ref{Section cartan homoto}. The former identity $[\bi_a,\bi_b]=0$ is trivial.
As regard the latter $[\bi_a,\bl_b]-\bi_{[a,b]}=0$, we
 recall that  $\bl_b=[d-t\Delta, \bi_b] -\bi_{db}$ (note that  the differential changes sign on the $k$-fold suspension). Moreover, we have the following explicit description
 \[
 \bl_b(c)=-\Delta(bc)+ (-1)^{\bar{b}} b \Delta(c).
 \]
Indeed,
\[
\bl_b(c)= [d-t\Delta, \bi_b](c) -\frac{(db)c}{t}
\]
\[
= (d-t\Delta)(\frac{bc}{t}) - (-1)^{\bar{b}}\bi_b(dc-t\Delta(c))-\frac{(db)c}{t}=
\]
\[ \frac{1}{t}(d(bc) -(-1)^{\bar{b}} b(dc) -(db)c     ) -\Delta(bc)+ (-1)^{\bar{b}} b \Delta(c). \]

Then, 
\[
[\bi_a,\bl_b](c)-\bi_{[a,b]}(c)=
\bi_a(-\Delta(bc)+ (-1)^{\bar{b}} b \Delta(c)) -(-1)^{\bar{a}(\bar{b}+1)}\bl_b(\frac{ac}{t}) -\frac{1}{t}[a,b]c\]
\[=\frac{1}{t}(-a\Delta(bc) + (-1)^{\bar{b}} ab \Delta(c) -(-1)^{\bar{a}(\bar{b}+1)} (-\Delta(bac)+ (-1)^{\bar{b}} b \Delta(ac))  \]
\[-(-1)^{\bar{a}}(\Delta(ab)c-\Delta(a)bc)+a\Delta(b)c)=0.
\]
The last equality  follows from the  the seven-terms relation satisfied by $\Delta$ (multiplying $(-1)^a t$).
\end{proof}

\begin{definition}
A dBV algebra $(A,d,\Delta)$  of degree $k$ has the \emph{degeneration property} if and only if for every $a_0 \in A$,
such that $d a_0=0$, there exists a sequence $a_i$, $i\geq 0$, with $\deg(a_i)=\deg(a_{i-1})-k-1$ and 
such that   
\[
 \Delta a_i= da_{i+1}, \qquad i\geq 0.
\] 
 \end{definition}

\begin{example}\label{example E1 degen implies dbV degener}
Let $(A, d, \Delta)$ be a dBV algebra and suppose that it is bigraded, i.e., $A= \bigoplus _{i,j \geq 0}A^{i,j}$ and 
 $d: A^{i+1,j}\to  A^{i,j}$ and $\Delta: A^{i,j}\to A^{i,j+1}$.
Then, the filtration $F_p = \oplus_{j\geq p}  A^{i,j}$ define a decreasing filtration of the double complex and therefore a spectral sequence.
If this spectral sequence degenerates at the first page $E_1$, then the dBV algebra $(A, d, \Delta)$ has the degeneration property 
\cite[Lemma 1.5]{Morgan}, \cite[Proposition 1.5]{DSV12}.

\end{example}

\begin{example}\label{example A((t)) dbV}

Let $(A, d, \Delta)$ be a dBV algebra.
On the complex $(A((t)), d-t\Delta)$,  consider the filtration $F^\bullet$,  defined by
$F^p=t^pA[[t]]$, for every $p \in \Z$.  Note that $A((t))=\bigcup _{p\in \Z} F^p$ and $F^0=A[[t]]$.
Then, the   dBV algebra $(A, d, \Delta)$ has the degeneration property
if and only if the morphism of complexes $(A[[t]],d-t\Delta) \to (A,d)$, given by $t\mapsto 0$ is surjective in cohomology, if and only if the inclusion of complexes  $(tA[[t]],d-t\Delta) \to  (A[[t]],d-t\Delta)$ is injective in cohomology.
In particular, the degeneration property  implies that the inclusion $F^p \to A((t))$ is injective in cohomology, for every $p$, and so  $A[[t]]\to A((t))$ is also   injective in cohomology.

\end{example}

 \begin{theorem}\label{theorem dbv degener implies homotopy abelian}
Let  $(A,d,\Delta)$ be  a dBV algebra with the degeneration property. Then, the associated DGLA $\g=A[k]
$ is homotopy abelian.
\end{theorem}

\begin{proof} 
According to the previous Lemma \ref{lem.cartanhomotopyforDBV}, it is well defined a Cartan homotopy
 $\bi\colon \mathfrak{g} \to\Hom^*_{\K}(A((t)),A((t)))$, whose associated Lie derivative has the following explicit expression
\[\bl_b(c)=-\Delta(bc)+ (-1)^{\bar{b}} b \Delta(c).
\]
 Therefore, considering the filtration $F^p=t^pA[[t]]$ of the  complex $(A((t)), d-t\Delta)$ as  in Example \ref{example A((t)) dbV},  we note that
\[\bi: \g \to \Hom^*(F^p,F^{p-1}) \qquad \mbox{and} \qquad  \bl: \g \to \Hom^*(F^p,F^p),\quad \forall \ p.\]
Next, consider the differential graded Lie algebra
\[
M=\Hom^*_{\K}(A((t)),A((t))),
\]
 the sub-DGLA
\[
N=\{\varphi \in M \mid \varphi(A[[t]]) \subset A[[t]]  \},
\]
and let $\chi\colon N \to M$ be the inclusion. Since $\bl(\g) \subset N$, according to
Lemma \ref{lem.cartan induce morfismo TW}, there exists  an induced  $L_{\infty}$-morphism $\psi:\g \dashrightarrow TW(\chi)$.

As observed in the Example  \ref{example A((t)) dbV}, the degeneration property  implies that the inclusion $A[[t]]\to A((t))$ is  injective in cohomology. Therefore,  the
DGLA $TW(\chi)$ is  homotopy abelian by Lemma \ref{lem.criterio TW abelian}. According to Lemma \ref{lem.criterioquasiabelianita}, to conclude the proof it is enough to show that $\psi$ induces an injective  morphism in cohomology.

As observed in Remark \ref{rem.quasiisoTWcono},   $TW(\chi)$ is quasi-isomorphic to
\[\coker(\chi)[-1]=\Hom^*_{\K}\left(A[[t]], \frac{A((t))}{A[[t]]}\right)[-1];
\]
therefore, it is sufficient to prove that the  morphism of complexes 
\[\bi\colon A\to \Hom^*_{\K}\left(A[[t]], \frac{A((t))}{A[[t]]}\right)[-k-1]\]
is injective in cohomology. It is actually enough 
 to prove the injectivity for the composition with the evaluation at $1\in A[[t]]$, i.e.,    the map 
\[  A\to \frac{A((t))}{A[[t]]},\qquad a\mapsto \frac{a}{t},\]  
is injective in cohomology. 
Note that this is equivalent  to the statement that the inclusion 
\[ \frac{F^{-1}}{F^0}\hookrightarrow \frac{A((t))}{F^0}\]
is injective in cohomology, since 
the map $a\mapsto \dfrac{a}{t}$ defines an isomorphism of DG-vector spaces $A\to F^{-1}/F^0$. The claim follows considering  the  short exact sequences  
\[ \xymatrix{ 0\ar[r]& F^0\ar[r]\ar@{=}[d]& F^{-1} \ar[d]_{j}\ar[r]& \dfrac{F^{-1}}{F^0}  \ar[d]\ar[r]&0\\
0\ar[r]& F^0\ar[r]& A((t)) \ar[r]& \dfrac{A((t))}{F^0}  \ar[r]&0\\
}
\]
and keeping in mind that the inclusion $j$ is injective in cohomology by the degeneration property   (Example  \ref{example A((t)) dbV}). 
\end{proof}

\begin{remark}
 The original proof of this theorem   (for $k=1$)  can be found in \cite[Theorem 1]{terilla} or \cite[Theorem 4.14]{KKP}. This proof was suggested to the author by Marco Manetti.
\end{remark}

\begin{example}  
\cite[Theorem 4.18]{KKP} Let X  be a compact  projective Calabi Yau variety of dimension $n$ over $\C$.
In this situation,  the relevant dBV algebra is $(A, d, \Delta)$ with
 $A= \Gamma  (X, \sA_X^{0,*}(  \wedge^\bullet \Theta_X))$, 
  $d= \debar$ and $\Delta =div_\omega= {i _ \omega}^{-1} \circ \de \circ  {i }_ \omega$. Here $\omega$ is a non vanishing section of
$\Omega^n_X$   and ${i }_ \omega:  \wedge^\bullet \Theta_X \to \Omega^{n-\bullet}_X   $ is the isomorphism given by the contraction with $\omega$. The contraction ${i }_ \omega$
gives an isomorphism of bicomplexes between the dBV algebra $(A, d, \Delta)$ 
and the   Dolbeault bicomplex $(A^{*,*}(X), \debar,\de)$. According to Example \ref{example E1 degen implies dbV degener}, the degeneration of the Hodge-to-de Rham   spectral sequence implies that $(A, d, \Delta)$ has the degeneration property.
Therefore, Theorem \ref{theorem dbv degener implies homotopy abelian} implies that  the associated DGLA  $L= \Gamma  (X, \sA_X^{0,*}(  \wedge^\bullet \Theta_X))$, 
 is homotopy abelian. The Kodaira Spencer DGLA of $X$  $\Gamma  (X, \sA_X^{0,*}(  \Theta_X))$
 is embedded in L and it is also an embedding in cohomology. According to Lemma \ref{lem.criterioquasiabelianita}, the Kodaira Spencer DGLA  is also homotopy abelian and   the deformations of $X$ are unobstructed.
\end{example}

\begin{example}

\cite[Section 4.3.3 (i)]{KKP}
 Let $X$ be a smooth projective n dimensional  variety over $\C$ and $D$ a smooth divisor, such that $\Omega^n(\log D)$ is trivial.
 In this case, the relevant dBV algebra is $(A, d, \Delta)$ with $A= \Gamma (X, \sA_X^{0,*} (\wedge^\bullet \Theta_X(-\log D) )$
 $d= \debar$ and $\Delta =div_\omega= {i _ \omega}^{-1} \circ \de \circ  {i }_ \omega$. Here $\omega$ is the non vanishing section of $\Gamma(X, \Omega^n_X(\log D))$ and ${i }_ \omega:  \wedge^\bullet \Theta_X(-\log D) \to \Omega^{n-\bullet}_X(\log D))   $ is the isomorphism given by the contraction with $\omega$. The map ${i }_ \omega$ identifies $(A,d, \Delta)$  with the logarithmic Dolbeault bicomplex $(A^{*,*}(\log D), \debar,\de)$. 
 Arguing as in the previous example and using the degeneration of the spectral sequence of Theorem \ref{teo degen deligne}, we can conclude that the DGLA   $(A_X^{0,*}( \Theta_X(-\log D)),  \debar, [ , ])$   is homotopy abelian and so the deformations of the pair $(X,D)$ are unobstructed  \cite[Lemma 4.19]{KKP}.


\end{example}

\begin{example}

\cite[Section 4.3.3 (ii)]{KKP}
 Let $X$ be a smooth projective n-dimensional Calabi Yau variety over $\C$ and $D$ a smooth divisor. In this case the relevant dBV  algebra  $(A, d, \Delta)$ is similar to the one introduced in  the previous example, indeed 
 $A= \Gamma (X, \sA_X^{0,*} (\wedge^\bullet \Theta_X(-\log D) )$ 
 $d= \debar$ and $\Delta =div_\omega= {i _ \omega}^{-1} \circ \de \circ  {i }_ \omega$. Here $\omega$ is a  non vanishing section of
$\Omega^n_X$. The degeneration of the spectral sequence of Theorem \ref{teo degen tensor} implies that 
 $(A, d, \Delta)$ has the degeneration property and so that      $(A_X^{0,*}( \Theta_X(-\log D)),  \debar, [ , ])$  is homotopy abelian  \cite[Lemma 4.20]{KKP}.
\end{example}

\end{document}